\documentclass[11pt,a4paper, final, twoside]{article}
\usepackage{amsmath}
\usepackage{fancyhdr}
\usepackage{amsthm}
\usepackage{amsfonts}
\usepackage{amssymb}
\usepackage{amscd}
\usepackage{latexsym}
\usepackage{graphicx}
\usepackage{graphics}
\usepackage{afterpage}
\usepackage[colorlinks=true, urlcolor=blue,  linkcolor=black, citecolor=black]{hyperref}
\usepackage{color}
\usepackage[T1]{fontenc}

\newcommand{\Z}{{\mathbf Z}}
\newcommand{\R}{\mathbf{R}}

\renewcommand{\P}{\mathrm{P}}
\newcommand {\E}{\mathrm{E}}
\newcommand {\C}{\mathrm{C}}

\renewcommand{\d}{\text{\rm d}}
\newcommand{\e}{\text{\rm e}}

\newcommand{\sA}{\mathcal{A}}

\newcommand{\lip}{\mathrm{Lip}}

\newtheorem{stat}{Statement}[section]
\newtheorem{proposition}[stat]{Proposition}

\newtheorem{theorem}[stat]{Theorem}

\newtheorem{lemma}[stat]{Lemma}
\theoremstyle{definition}
\newtheorem{definition}[stat]{Definition}\newtheorem{remark}[stat]{Remark}

\newtheorem{condition}[stat]{Condition}

\numberwithin{equation}{section}
\begin{document}

\hyphenpenalty=100000

\begin{flushright}

{\Large \textbf{On Space-Time Fractional Heat Type Non-Homogeneous Time-Fractional Poisson Equation}}\\[5mm]
{\textbf{Ejighikeme McSylvester Omaba}}$^\mathrm{1^*}$\\[3mm]
\footnote{\emph{*Corresponding author: E-mail: mcsylvester.omaba@funai.edu.ng;}}$^\mathrm{1}${\footnotesize \it Department of Mathematics, Computer Science, Statistics and Informatics, Faculty of Science, Federal University Ndufu-Alike Ikwo, Ebonyi State, Nigeria.}
\end{flushright}

\begin{abstract} Consider the following space-time fractional heat equation with Riemann-Liouville derivative of non-homogeneous time-fractional Poisson process 
\begin{eqnarray*}
\partial^\beta_t u(x,t)=-\kappa(-\Delta)^{\alpha/2} u(x,t)+I_t^{1-\beta}[\sigma(u)D_t^\vartheta N^\nu_\lambda(t)],\,\,t\geq 0,\,x\in\R^d,
\end{eqnarray*}where $\kappa>0,\,\,\beta,\,\vartheta\in(0,1),\,\,\nu\in(0,1],\,\alpha\in(0,2].$ The operator
 $D_t^\vartheta N^\nu_\lambda(t)=\frac{\d}{\d t}I_t^{1-\vartheta}N_\lambda^\nu(t)=\frac{\d}{\d t}\mathcal{N}_\lambda^{1-\vartheta,\nu}(t)$ with $\mathcal{N}_\lambda^{1-\vartheta,\nu}(t)$ the  Riemann-Liouville non-homogeneous fractional integral process, $\partial^\beta_t$ is the Caputo fractional derivative, $-(-\Delta)^{\alpha/2}$ is the generator of an isotropic stable process, $I^\beta_t$ is the fractional integral operator, and $\sigma : \R \rightarrow \R$ is Lipschitz continuous. The above time fractional stochastic heat type equations may be used to model sequence of catastrophic events with thermal memory. The mean and variance for the process $\frac{\d}{\d t}\mathcal{N}^{1-\vartheta,\nu}_\lambda(t)$ for some specific rate functions were computed. Consequently, the growth moment bounds for the class of heat equation perturbed with the non-homogeneous fractional time Poisson process were given and we show that the solution grows  exponentially for some  ~small ~time~ interval $t\in [t_0,T],\,\,T<\infty$~ and $t_0>1$;\,\, ~that is, ~the ~result ~establishes ~that ~the ~energy ~of ~the ~~solution~ ~grows ~~atleast as  
$ c_4(t+t_0)^{(\vartheta-a\nu)}\exp(c_5 t)$ and at most as $c_1 t^{(\vartheta- a\nu)}\exp(c_3 t)$ for different conditions on the initial data, where $c_1,\,c_3,\,c_4$ and $c_5$ are some positive constants depending on $T$. Existence and uniqueness result for the mild solution to the equation was given under linear growth condition on $\sigma$.  
\end{abstract}

\section{Introduction}
The authors in \cite{Foondun},\,\,\cite{Mijena}, considered the following equations
\begin{eqnarray*}
\partial^\beta_t u(x,t)=-\kappa(-\Delta)^{\alpha/2}u(x,t)+I_t^{1-\beta}[\lambda\sigma(u)\dot{W}(t,x)],
\end{eqnarray*}
\begin{eqnarray*}
\partial^\beta_t u(x,t)=-\kappa(-\Delta)^{\alpha/2}u(x,t)+I_t^{1-\beta}[\lambda\sigma(u)\dot{F}(t,x)],
\end{eqnarray*}
in $(d+1)$ dimensions, where $\kappa>0,\,\beta\in (0, 1),\,\alpha\in (0, 2]$ and $d < \min\{2,\beta^{-1}\}\alpha,\,\partial^\beta_t$ is the Caputo fractional derivative, $-(-\Delta)^{\alpha/2}$ is the generator of an isotropic stable process, $I^\beta_t$ is the fractional integral operator, $\dot{W} (t, x)$ is space-time white noise, and $\sigma : \R \rightarrow \R$ is Lipschitz continuous. 
See \cite{Foondun},\,\,\cite{Mijena} for the formulation of solutions of the above equations and \cite{Umarov},\,\,\cite{Umarov1} for the use of {\bf time fractional Duhamel's principle} and how to remove the operator $\partial_t^{1-\beta}$ term appearing in the solution by defining a fractional integral operator $I^{1-\beta}_t$. We now attempt to define the equivalent equation for Riemann-Liouville derivative of non-homogeneous time-fractional Poisson process 
\begin{eqnarray}\label{eqn:white}
\partial^\beta_t u(x,t)=-\kappa(-\Delta)^{\alpha/2} u(x,t)+I_t^{1-\beta}[\sigma(u)D_t^\vartheta N^\nu_\lambda(t)],\,\,t\geq 0,\,x\in\R^d,
\end{eqnarray}
where $D_t^\vartheta N^\nu_\lambda(t)=\frac{\d}{\d t}I_t^{1-\vartheta}N_\lambda^\nu(t)=\frac{\d}{\d t}\mathcal{N}_\lambda^{1-\vartheta,\nu}(t)$ with $\mathcal{N}_\lambda^{1-\vartheta,\nu}(t)$ the  Riemann-Liouville non-homogeneous fractional integral process studied by Orsingher and Polito \cite{Orsingher} and $I_t^{1-\vartheta}$ is the fractional integral operator. We therefore study some growth bounds for the above space-time fractional heat equation with Riemann-Liouville derivative of non-homogeneous time-fractional Poisson process with Caputo derivatives. The mean and variance for the process for some specific rate functions were computed and consequently the moment growth bounds were estimated, and we conclude that the solution (or the energy of the solution) grows in time at most a precise exponential rate at some small time interval.\\[3mm]
\noindent
The fractional Poisson process is a generalisation of the standard Poisson process. The use of fractional Poisson process has received serious interest for almost two decades now. The process was first introduced and studied by Repin and Saichev \cite{Repin}, followed by Laskin \cite{Laskin} and many others like Mainardi and his co-authors \cite{Gorenflo},\,\,\cite{Gorenflo1},\,\,\cite{Mainardi},\,\,\cite{Mainardi1},  Beghin and Orsingher \cite{Beghin},\,\,\cite{Beghin1},\,\,\cite{Orsingher} and its representation in terms of stable subordinator \cite{Gorenflo},\,\,\cite{Gorenflo1},\,\,\cite{Leonenko},\,\,\cite{Meerschaert} and \cite{Silva}.
See the above papers and their references for a complete study on fractional Poisson process and its fractional distributional properties. See a recent article \cite{Maheshwari} on non-homogeneous fractional Poisson processes which involves replacing the time parameter in the fractional Poisson process with some suitable function of time and also some numerical (or modelling) applications in \cite{Esen},\,\,\cite{Gomez} and \cite{Tasbozan}. The physical motivation to studying the above time-fractional SPDEs is that they may arise naturally in modelling sequence of catastrophic events with thermal memories; example, in  hydrology and Seismology, it may be used to model earthquake inter-arrival times, \cite{Biard},\,\,\cite{Crescenzo},\,\,\cite{Laskin1},\,\,\cite{Repin} and \cite{Scalas}. Let $\gamma>0$ and define the fractional integral by $$I^\gamma f(t)=\frac{1}{\Gamma(\gamma)}\int_0^t(t-s)^{\gamma-1}f(s)\d s.$$
  The Caputo time-fractional derivative is given by 
 \begin{eqnarray*}
D^\beta_* u(x,t)=\left \{
\begin{array}{lll}
 \frac{1}{\Gamma(m-\beta)}\int_0^t\frac{u^{(m)}(x,s)}{(t-s)^{\beta+1-m}}\d s,\,\,m-1<\beta<m,\\
\frac{\d^m}{\d t^m}u(x,t),\,\,\beta=m,
  \end{array}\right.
\end{eqnarray*}
and with $m=1$, we denote the Caputo derivative of order $\beta\in(0,1)$  by:
$$\partial_t^\beta u(x,t)= \frac{1}{\Gamma(1-\beta)}\int_0^t\partial_s u(x,s)\frac{\d s}{(t-s)^{\beta}}.$$ 
For $1-\beta\in(0, 1)$ and $g\in L^\infty(\R_+)$ or $g\in C(\R_+)$, then $$\partial^{1-\beta}_t I^{1-\beta}_t g(t) = g(t).$$
 We also define a Riemann-Liouville time-fractional derivative  by 
 \begin{eqnarray*}
D^\vartheta f(t)=\left \{
\begin{array}{lll}
 \frac{\d^m}{\d t^m}\bigg[\frac{1}{\Gamma(m-\vartheta)}\int_0^t\frac{f(s)}{(t-s)^{\vartheta+1-m}}\d s\bigg],\,\,m-1<\vartheta<m,\\
\frac{\d^m}{\d t^m}f(t),\,\,\vartheta=m.
  \end{array}\right.
\end{eqnarray*}
Now to make sense of the derivative $D_t^\vartheta f(t):=D^\vartheta f(t)$ for $m=1$ and $\vartheta\in(0,1)$, that is, for $D_t^\vartheta f(t)=D^1_t I^{1-\vartheta}_t f(t),$ we state the following theorem:
\begin{theorem}
Let $f(x,t)$ be a well-behaved function such that the partial derivative of $f$ with respect to $t$ exists and is continuous. Then
$$\frac{\d}{\d t}\bigg(\int_{a(t)}^{b(t)}f(x,t)\d x\bigg)=\int_{a(t)}^{b(t)}\partial_t f(x,t)\d x+f(b(t),t).b'(t)-f(a(t),t).a'(t).$$ 
\end{theorem}
\begin{remark}Let $f(s,t)=(t-s)^{-\vartheta} N^\nu_\lambda(s)$, then applying the above theorem,we have
\begin{eqnarray*}
D^\vartheta_t N^\nu_\lambda(t)=\frac{1}{\Gamma(1-\vartheta)}\frac{\d}{\d t}\int_0^tf(s,t)\d s&=&\frac{1}{\Gamma(1-\vartheta)}\int_0^t \partial_t f(s,t)\d s\\&=&\frac{-\vartheta}{\Gamma(1-\vartheta)}\int_0^t(t-s)^{-\vartheta-1} N^\nu_\lambda(s)\d s.
\end{eqnarray*}
\end{remark}
\noindent
Define the mild solution to equation \eqref{eqn:white} in sense of Walsh \cite{Walsh} by following similar step in \cite{Foondun},\,\,\cite{Mijena} as follows:
\begin{definition}
We say that a process $\{u(x,t)\}_{x\in \R^d, t>0}$ is a mild solution of 
\eqref{eqn:white} if a.s, the following is satisfied
\begin{eqnarray}\label{mild:white}
u(x,t)&=&\int_{\R^d} G_{\alpha,\beta}(t,x-y)u_0(y)\d y\\&+&\int_0^t\int_{\R^d} G_{\alpha,\beta}(t-s,x-y)\sigma(u(s,y))D_s^\vartheta N^\nu_\lambda(s)\d y\d s\nonumber\\
&=&\int_{\R^d} G_{\alpha,\beta}(t,x-y)u_0(y)\d y+\int_0^t\int_{\R^d} G_{\alpha,\beta}(t-s,x-y)\sigma(u(s,y))\nonumber\\&\times&\bigg\{\frac{-\vartheta}{\Gamma(1-\vartheta)}\int_0^s(s-\tau)^{-\vartheta-1} N^\nu_\lambda(\tau)\d\tau\bigg\}\d y\d s,\nonumber
\end{eqnarray}
where $ G_{\alpha,\beta}(t,x)$ is the time-fractional heat kernel. 
 If in addition to the above, $\{u(x,t)\}_{x\in \R^d, t>0}$ satisfies the following condition
\begin{equation}\label{r-field}
\sup_{0\leq t\leq T}\sup_{x\in \R^d}\E|u(x,t)|<\infty,
\end{equation}
for all $T>0$, then we say that $\{u(x,t)\}_{x\in \R, t>0}$ is a random field solution to \eqref{eqn:white}
with the following norm: $$\|u\|_{1,\beta}=\sup_{t\in[0,T]}\sup_{x\in\R^d}\e^{-\beta t}\E|u(x,t)|,\,\,\textrm{for}\,\,\beta>0.$$
\end{definition}
\noindent
The paper is outlined as follows. Section 2 gives the summary statement of theorems of the main results. Section 3 surveys some basic preliminary concepts, including estimates on the mean and variance of the Riemann-Liouville fractional integral process and its non-homogeneous counterpart. Some auxiliary results for existence and uniqueness result were obtained in section 4, and the energy moment growth estimates, proofs of main results and conclusion of the results given in section 5.
\section{Main Results}We assume the following condition on $\sigma$; which says essentially that $\sigma$ is globally Lipschitz:
\begin{condition}\label{cond:E-U}
There exists a finite positive constant, $\lip_\sigma$ such that for all $x,\,y\in \R$, we have 
\begin{equation*}\label{cond:sigma}
 |\sigma(x)-\sigma(y)|\leq \lip_\sigma|x-y|.
\end{equation*}
We will take $\sigma(0)=0$ for convenience.
\end{condition}
\noindent
For the lower bound result, we require the following extra condition on $\sigma$:
\begin{condition}\label{cond:L-B}There exists a finite positive constant, $L_\sigma$ such that for all $x\in\R$, we have, 
$$\sigma(x)\geq L_\sigma|x|.$$
\end{condition}
\noindent
Here, we give the statements of our main results. The first result follows by assuming that the non-random initial data $u_0:\R^d\rightarrow\R$ is a non-negative bounded function.
\begin{theorem}\label{growth-upperbound}
Given that condition \ref{cond:E-U} holds and $u_0$ bounded above, then there exists $t_0>1$ such that for all $t_0<t<T<\infty$, we have
 $$\sup_{x\in\R^d}\E|u(x,t)|\leq c_1 t^{\vartheta-\nu}\exp(c_3 t),$$ with $c_1=c T^{\nu-\vartheta+\frac{\beta}{\alpha}(1-d)},\,\,\textrm{and}\,\,c_3=\frac{\lambda\lip_\sigma c_2}{\Gamma(1-\vartheta+\nu)} T^{\nu-\vartheta+\frac{\beta}{\alpha}(1-d)}.$
\end{theorem}
\noindent
Next we drop the assumption that the initial condition $u_0$ is bounded above and assume that $u_0(x)$ is positive: 
\begin{definition}
The initial function $u_0$ is assumed to be a bounded non-negative function such that  $$\int_A u_0(x)\d x>0,\,\,\text{for some}\,\, A\subset\R^d.$$ 
That is, we define $u_0$ as any measurable function $u_0:\R^d\rightarrow\R_+$ which is positive on a set of positive measure. This assumption implies that the set $A=\big\{x:u_0(x)>\frac{1}{n}\big\}\subset\R^d$  has positive measure for all but finite many $n$. Thus by Chebyshev's inequality, $$\int_{\R^d} u_0(x)\d x\geq \int_{\{x:\,u_0(x)>\frac{1}{n}\}} u_0(x)\d x\geq \frac{1}{n}\mu\bigg\{x:u_0(x)>\frac{1}{n}\bigg\}>0,$$ where $\mu$ is a Lebesgue measure.
\end{definition}
\noindent
 Therefore with the assumption that the initial condition $u_0$ is positive on a set of positive measure, we then have the following lower bound estimate:
\begin{theorem}\label{growth-firstmoment} Suppose that condition \ref{cond:L-B} together with $\|u_0\|_{L^1(B(0,1))}>0$ hold. Then there exists $t_0>1$ such that for all $t_0<t<T<\infty$, we have
$$\inf_{x\in B(0,1)}\E|u(x,t)|\geq c_4(t+t_0)^{\vartheta-\nu}\exp(c_5 t),\,\,\text{for all}\,\, t\in[t_0, T],$$ where $c_4=  c_1(T+t_0)^{-\big\{{\beta d}/\alpha+\vartheta-\nu\big\}},$ and $c_5=\frac{\lambda L_\sigma c_3}{\Gamma(1-\vartheta+\nu)} (T+t_0)^{\nu-\vartheta} T^{-\beta/\alpha}$. 
\end{theorem}
\noindent
We also give equivalent results for the non-homogeneous fractional time process for the Weibull's rate function as follow:
\begin{theorem}\label{growth-upperbound-non}
Given that condition \ref{cond:E-U} holds and $u_0$ bounded above, then there exists $t_0>1$ such that for all $t_0<t<T<\infty$, we have
 $$\sup_{x\in\R^d}\E|u(x,t)|\leq c_1 t^{\vartheta-a\nu}\exp(c_3 t),$$ with $c_1= c T^{a\nu-\vartheta+\frac{\beta}{\alpha}(1-d)},\,\,\textrm{and}\,\,c_3=\frac{b^{-a\nu}\lip_\sigma c_2}{\Gamma(\nu+1)}\frac{\Gamma(1+a\nu)}{\Gamma(1-\vartheta+a\nu)} T^{a\nu-\vartheta+\frac{\beta}{\alpha}(1-d)}.$
\end{theorem}
\begin{theorem}\label{growth-firstmoment-non} Suppose that condition \ref{cond:L-B} together with $\|u_0\|_{L^1(B(0,1))}>0$ hold. Then there exists $t_0>1$ such that for all $t_0<t<T<\infty$, we have
$$\inf_{x\in B(0,1)}\E|u(x,t)|\geq c_4(t+t_0)^{\vartheta-a\nu}\exp(c_5 t),\,\,\text{for all}\,\, t\in[t_0, T],$$ where $c_4=  c_1(T+t_0)^{-\big\{{\beta d}/\alpha+\vartheta-a\nu\big\}},\,\,\textrm{and}\,\,c_3=\frac{b^{-a\nu}L_\sigma c_3}{\Gamma(\nu+1)}\frac{\Gamma(1+a\nu)}{\Gamma(1-\vartheta+a\nu)}(T+t_0)^{a\nu-\vartheta} T^{-\beta/\alpha}$.
\end{theorem}
\section{Preliminaries}
Consider the following fractional diffusion equation
\begin{eqnarray*}
\partial^\beta_t u(x,t)=-\kappa(-\Delta)^{\alpha/2}u(x,t),
\end{eqnarray*}with initial condition $u(x,0)=u_0(x)$. Given that the solution is 
$$u(x,t)=\int_{\R^d}G_{\alpha,\beta}(t,x-y)u_0(y)\d y,$$ and suppose that $G_{\alpha,\beta}(t,x)$ is the fundamental solution of the fractional heat type equation $$\partial^\beta_t G_{\alpha,\beta}(t,x)=-\kappa(-\Delta)^{\alpha/2}G_{\alpha,\beta}(t,x).$$
Take Laplace transform in the time variable and Fourier transform in the space variable of both sides of the above equation as follows:
$$s^\beta \widehat{\widetilde{G}}_{\alpha,\beta}(\xi,s)-s^{\beta-1}=-\kappa|\xi|^\alpha \widehat{\widetilde{G}}_{\alpha,\beta}(\xi,s),$$ which follows that $$ \widehat{\widetilde{G}}_{\alpha,\beta}(\xi,s)=\frac{s^{\beta-1}}{s^\beta+\kappa|\xi|^\alpha}.$$ Now take inverse Laplace transform in $s$, we have 
$$\widehat{G}_{\alpha,\beta}(\xi,t)=E_\beta(-\kappa|\xi|^\alpha t^\beta),$$ where $$E_\beta(x)=\sum_{n=0}^\infty\frac{x^n}{\Gamma(n\beta+1)},$$ is the Mittag-Leffler function with the following uniform estimate $$\frac{1}{1+\Gamma(1-\beta)x}\leq E_\beta(-x)\leq\frac{1}{1+\Gamma(1+\beta)^{-1}x}\,\,\textrm{for} \,\,x>0.$$ Next, take inverse Fourier transform in $\xi$, $$G_{\alpha,\beta}(t,x)=\frac{1}{2\pi}\int_{-\infty}^{+\infty}\e^{-i x\xi}E_\beta(-\kappa|\xi|^\alpha t^\beta)\d\xi.$$
We now make use of the following property of integral
$$\frac{1}{2\pi}\int_{-\infty}^{+\infty}\e^{-i x\xi}f(x\xi)\d\xi=\frac{1}{\pi}\int_0^{+\infty}f(x\xi)\cos(x\xi)\d\xi.$$ Therefore 
$$G_{\alpha,\beta}(t,x)=\frac{1}{\pi}\int_0^{+\infty}E_\beta(-\kappa|\xi|^\alpha t^\beta)\cos(x\xi)\d\xi.$$
For $\alpha=2,\,\,\beta=1$,
 $$G_{2,1}(t,x)=\frac{1}{\pi}\int_0^{+\infty}E_1(-\kappa\xi^2 t)\cos(x\xi)\d\xi=\frac{1}{\pi}\frac{\e^{-\frac{x^2}{4\kappa t}\sqrt{\pi}}}{2\sqrt{\kappa t}}=\frac{1}{\sqrt{4\kappa\pi t}}\exp(-\frac{x^2}{4\kappa t}).$$
 Let $X_t$  be a symmetric $\alpha$-stable process on $\R^d$ whose transition density $p(t,x)$, relative to Lebesgue measure, uniquely determined by its Fourier transform is: $$\E[\exp(i\xi X_t)]=\int_{\R^d} \e^{-i x\,\xi}p(t,\,x)\d x=\e^{-t\kappa |\xi|^\alpha},\quad \xi\in\R^d.$$ Let $\{D_\beta(t)\}_{t\geq0}$ be the $\beta$-stable subordinator with Laplace transform $\E[\e^{-s D_\beta(t)}]=\e^{-t s^\beta}$, or inverse stable subordinator of index $\beta$ and $E_t$ its first passage time. Given that the density of $E_t$ is  $$f_{E_t}(x)=t\beta^{-1}x^{-1-1/\beta}g_\beta(t x^{-1/\beta}),$$ with $g_\beta(.)$ the density function of $D_\beta(1)$, then the density $G_{\alpha,\beta}(t,x)$ of the time changed process $X_{E_t}$ is given by
 $$G_{\alpha,\beta}(t,x)=\int_0^\infty p(s,x)f_{E_t}(s)\d s.$$
\noindent
We now present some properties of $p(t,x)$, see \cite{Sugitani}, that will be needed to prove estimates on $G_{\alpha,\beta}(t,x)$.
\begin{eqnarray}\label{ker}
p(t,x)&=&t^{-d/\alpha}p(1,t^{-1/\alpha}x)\nonumber\\
 p(st,x)&=&t^{-d/\alpha}p(s,t^{-1/\alpha}x).
 \end{eqnarray}
 From the above relation, $p(t,0)=t^{-d/\alpha}\,p(1,0)$, is a decreasing function of $t$.  The heat kernel $p(t,x)$ is also a decreasing function of $|x|$, that is,
$$|x|\geq|y|\,\,\,\,\,\textrm{implies that}\,\,\,\,\, p(t,x)\leq p(t,y).$$ This and equation \eqref{ker} imply that for all $t\geq s$,
 \begin{eqnarray*}
 p(t,x)=p(t,|x|)&=&p\bigg(s.\frac{t}{s},|x|\bigg)=\bigg(\frac{t}{s}\bigg)^{-d/\alpha}p\bigg(s,\bigg(\frac{t}{s}\bigg)^{-1/\alpha}|x|\bigg)\\
 &\geq&\bigg(\frac{s}{t}\bigg)^{d/\alpha}p(s,|x|)\qquad \bigg(\textrm{since}\, \bigg(\frac{t}{s}\bigg)^{-1/\alpha}|x|\leq |x|\bigg)\\
 &=&\bigg(\frac{s}{t}\bigg)^{d/\alpha}p(s,x).
 \end{eqnarray*}  
 \begin{proposition}\label{prop:alpha-stable}
Let $p(t,x)$ be the transition density of a strictly $\alpha$-stable process. If $p(t,0)\leq1$ and $a\geq2$, then
\begin{equation*}
p\big(t,\frac{1}{a}(x-y)\big)\geq p(t,x)p(t,y),\,\,\,\forall\,x,\,y\in\R^d.
\end{equation*}
\end{proposition}
\begin{proof}
Given that $$\frac{1}{a}|x-y|\leq\frac{2}{a}|x|\vee\frac{2}{a}|y|\leq|x|\vee|y|,$$ then it follows from the above that,
\begin{eqnarray*}
p\big(t,\frac{1}{a}(x-y)\big)\geq p(t,|x|\vee|y|)
&\geq&  p(t,|x|)\wedge  p(t,|y|)\\&\geq& p(t,|x|) p(t,|y|)=p(t,x) p(t,y)
\end{eqnarray*}
\end{proof}
\noindent
The transition density also satisfies the following Chapman-Kolmogorov equation, $$\int_{\R^d}p(t,x)p(s,x)\d x= p(t+s,0).$$

\begin{lemma}\cite{Sugitani}
Suppose that $p(t,x)$ denotes the heat kernel for a strictly stable process of order $\alpha$.  Then the following estimate holds:
\begin{equation*}
p(t,x,y) \asymp t^{-d/\alpha}\wedge \frac{t}{|x-y|^{d+\alpha}},\quad\text{for\,\,all}\quad t>0\quad\text{and} \quad x,\,y\in \R^d.
\end{equation*}
Here and in the sequel, for two non-negative functions $f,g,\,\,\,f\asymp g$ means that there exists a positive constant $c>1$ such that $c^{-1}g\leq f\leq c\, g$ on their common domain of definition.
\end{lemma}
\noindent
We now state the following estimate on $G_{\alpha,\beta}(t,x)$ whose proof in \cite{Foondun} employ the above properties of heat kernel of $\alpha$-stable process.
\begin{lemma}\cite{Foondun}\label{Caputo-fractional} (a) There exists a positive constant $c_1$ such that for all $x\in\R^d$,
$$G_{\alpha,\beta}(t,x)\geq c_1\bigg(t^{-\frac{\beta d}{\alpha}}\wedge\frac{t^\beta}{|x|^{d+\alpha}}\bigg).$$
(b) If we further suppose that $\alpha>d$, then there exists a positive constant $c_2$ such that
$$G_{\alpha,\beta}(t,x)\leq c_2\bigg(t^{-\frac{\beta d}{\alpha}}\wedge\frac{t^\beta}{|x|^{d+\alpha}}\bigg).$$
\end{lemma}
\subsection{Homogeneous Fractional Poisson process}
 For the standard Poisson process $\{N(t)\}_{t\geq0}$ with intensity $\lambda>0$, the probability distribution satisfies the following difference-differential equation, see \cite{Beghin},\,\,\cite{Beghin1} and \cite{Silva},
 $$\frac{d}{d t}p(n,t)=-\lambda\big(p(n,t)-p(n-1,t)\big),\,\,\,n\geq 1,$$ with $p_n(0)=0$ if $n=0$ and is zero for $n\geq 1$. The solution is given by 
 $$p(n,t)=\P[N(t,\lambda)=n]=\frac{(\lambda t)^n e^{-\lambda t}}{n!}.$$ The waiting time distribution function for the process is given by $\phi(t)=\lambda \e^{-\lambda t},\,\,\lambda>0,\,\,t\geq0$ and its moment generating function given by $$\E[e^{s N(t)}]=\exp(\lambda t(\e^s-1)),\,\,\,s\in\R.$$ 
 \begin{definition}(Fractional Poisson process) Fractional Poisson process is a renewal process with inter-times between events represented by Mittag-Leffler distributions,\,\,see \cite{Beghin1},\,\,\cite{Biard} and \cite{Silva}. The fractional Poisson process $N^\nu(t),\,\,0<\nu\leq 1$ satisfies 
 \begin{eqnarray}D^\nu_t p_\nu(n,t)&=&-\lambda\big(p_\nu(n,t)-p_\nu(n-1,t)\big),\nonumber\\
 D^\nu_t p_\nu(0,t)&=&-\lambda p_\nu(0,t),
 \end{eqnarray} with $ p_\nu(n,0)=1$ if $n=0$ and zero for $n\geq 1$. The symbol $D^\nu_t$ denotes the fractional derivative in the sense of Caputo-Dzhrbashyan, defined by
  \begin{eqnarray*}
D^\nu_t f(t)=\left \{
\begin{array}{lll}
\frac{1}{\Gamma(1-\nu)}\int_0^t\frac{f'(s)}{(t-s)^\nu}d s,\,\,0<\nu<1,\\
\\
f'(t),\,\,\nu=1.
\end{array}\right.
\end{eqnarray*} The solution is given by $$ p_\nu(n,t)=\P[N^\nu(t)=n]=\frac{(\lambda t^\nu)^n}{n!}E_{\nu,1}^{(n)}(-\lambda t^\nu)=\frac{(\lambda t^\nu)^n}{n!} \sum_{k=0}^\infty \frac{(n+k)!}{k!}\frac{(-\lambda t^\nu)^k}{\Gamma\big(\nu(k+n)+1\big)}.$$ Its waiting time distribution function is given by $\phi_\nu(t)=\lambda t^{\nu-1}E_{\nu,1}(-\lambda t^\nu)$ where
 $$E_{\alpha,\beta}(z)=\sum_{k=0}^\infty\frac{z^k}{\Gamma(k\alpha+\beta)},\,\,\alpha,\beta\in\C,\,\mathcal{R} (\alpha),\mathcal{R}(\beta)>0,\,\,z\in\R,$$ is the Mittag-Leffler function. 
   \end{definition}
 \begin{theorem}\label{theorem:homo-frac}\cite{Maheshwari}
   Consider the fractional Poisson process $\{N^\nu(t)\}_{t\geq0},\,\,\nu\in(0,1]$. The moment generating function of the process $N^\nu(t)$ can be expressed as follows:
 $$\E[e^{s N^\nu(t)}]=E_{\nu,1}(\lambda(e^s-1)t^\nu),\,\,\,s\in\R.$$  The mean and the variance of $N^\nu(t)$ are given by 
  $$\E[N^\nu(t)]=\frac{\lambda t^\nu}{\Gamma(\nu+1)},\,\,\,\,\,
 Var[N^\nu(t)]=\frac{2(\lambda t^\nu)^2}{\Gamma(2\nu+1)}-\frac{(\lambda t^\nu)^2}{\Gamma^2(\nu+1)}+\frac{\lambda t^\nu}{\Gamma(\nu+1)}.$$
  In general, the $p$th order moment of the fractional process is given by
 $$\E[N^\nu(t)]^p=\sum_{k=0}^\infty S_\nu(p,k)(\lambda t^\nu)^k,$$ where $S_\nu(p,k)$ is a fractional Stirling number.
   \end{theorem}
 \subsection{Non-homogeneous fractional Poisson process}  The non-homogeneous fractional Poisson process is obtained by replacing the time variable in the fractional Poisson process of renewable type with an appropriate function of time - $\Lambda(t)$. 
 \begin{definition}(Non-homogeneous Poisson process) A counting process $\{N_\lambda(t)\}_{t\geq0}$ is said to be a non-homogeneous Poisson process with intensity function $\lambda(t):[0,\infty)\rightarrow[0,\infty)$ if $$N_\lambda(t)=N(\Lambda(t)),\,\,t\geq0.$$
  The non-homogeneous Poisson process is specified either by its intensity function $\lambda(t)$ or more generally by its expectation function $\Lambda(t)=\E[N_\lambda(t)]$. When the intensity function $\lambda(t)$ exists, one denotes $$\Lambda(t,s)=\int_s^t\lambda(y)\d y,$$ where the  function $\Lambda(t)=\Lambda(0,t)$ is known as the rate function or cumulative rate function. The stochastic process $N_\lambda(t)$ has an independent but not necessarily stationary increments: let $0\leq s<t$, then the Poisson marginal distributions of $N_\lambda$ is given by $$\P[N_\lambda(t+s)-N_\lambda(t)=n]=\frac{e^{-(\Lambda(t+s)-\Lambda(t))}(\Lambda(t+s)-\Lambda(t))^n}{n!},\,\,n\in\Z_+.$$
  \end{definition}
 \begin{remark} The following are some examples of rate functions:
 \begin{itemize}
 \item Weibull's rate function: $$\Lambda(t)=\bigg(\frac{t}{b}\bigg)^a,\,\,\lambda(t)=\frac{a}{b}\bigg(\frac{t}{b}\bigg)^{a-1},\,\,a\geq0,\,b>0,$$
 \item Gompertz's rate function: $$\Lambda(t)=\frac{a}{b}e^{b t}-\frac{a}{b},\,\,\lambda(t)=a e^{b t},\,a,b>0,$$
 \item Makeham's rate function: $$\Lambda(t)=\frac{a}{b}e^{b t}-\frac{a}{b}+\mu t,\,\,\lambda(t)=a e^{b t}+\mu,\,\,a>0,\,b>0,\,\mu\geq0.$$
 \end{itemize}
 \end{remark}
 \begin{definition}(Non-homogeneous fractional Poisson process) The non-homogeneous fractional Poisson process is defined as $$N_\lambda^\nu(t)=N^\nu(\Lambda(t)),\,\,t\geq0,\,\,0<\nu\leq 1,$$ where $N^\nu(t)$ is the fractional Poisson process and $\Lambda(t)$ is the rate (or cumulative rate) function.
 \end{definition}
\noindent
One observes that when $\lambda(t)=\lambda^{1/\nu},\,\,t\geq0$ then $\Lambda(t)=\lambda^{1/\nu} t$ and the non-homogeneous fractional Poisson process easily gives the fractional Poisson process. The probability mass function of the non-homogeneous fractional Poisson process is given by
 $$p_\nu(n,\Lambda(t))=\P[N_\lambda^\nu(\Lambda(t))=n]=\frac{(\Lambda(t))^{n\nu}}{n!} \sum_{k=0}^\infty \frac{(n+k)!}{k!}\frac{(-\Lambda(t)^\nu)^k}{\Gamma\big(\nu(k+n)+1\big)}.$$
  \begin{theorem}\label{theorem:non-homo-frac}\cite{Maheshwari} Let $0<s\leq t<\infty,\,\,\,q=1/\Gamma(1+\nu)$ and $d=\nu q^2 B(\nu,1+\nu)$ then the mean and variance of the process $N_\lambda^\nu(t)$ are given by
 $$\E[N_\lambda^\nu(t)]=q\Lambda^\nu(t),\,\,Var[N_\lambda^\nu(t)]=q\Lambda^\nu(t)\big(1-q\Lambda^\nu(t)\big)+2 d \Lambda^{2\nu}(t),$$ where $B(a,b)$ is a Bessel function.
 \end{theorem}
\noindent
 We now return to equation \eqref{eqn:white} and compute the expectation of $\mathcal{N}^{1-\vartheta,\nu}(t)$ for the fractional Poisson process $N^\nu(t)$ and $\mathcal{N}^{1-\vartheta,\nu}_\lambda(t)$ for the non-homogeneous fractional Poisson process using some specific rate functions.
 \begin{lemma}\label{lem:frac}
 Consider the Riemann-Liouville fractional integral process $\mathcal{N}^{1-\vartheta,\nu}(t)$, $0<1-\vartheta<1$ and $0<\nu\leq 1$, then we have 
 \begin{equation*}
 \E[D^\vartheta_t N^\nu(t)]=\E[\frac{\d}{\d t}\mathcal{N}^{1-\vartheta,\nu}(t)]=\frac{\lambda t^{\nu-\vartheta}}{\Gamma(1+\nu-\vartheta)}.
 \end{equation*}
 \end{lemma}
 \begin{proof}From Theorem \ref{theorem:homo-frac}, we have that
 \begin{eqnarray*}
\E[D^\vartheta_t N^\nu(t)]&=& \E[\frac{\d}{\d t}\mathcal{N}^{1-\vartheta,\nu}(t)]\\&=&\frac{-\vartheta}{\Gamma(1-\vartheta)}\int_0^t(t-s)^{-\vartheta-1}\E[N^\nu(s)]\d s\\
 &=&\frac{-\beta}{\Gamma(1-\vartheta)}\int_0^t(t-s)^{-\vartheta-1}\frac{\lambda s^\nu}{\Gamma(\nu+1)}\d s\\
 &=&\frac{-\vartheta\lambda}{\Gamma(1-\vartheta)\Gamma(\nu+1)}\int_0^t(t-s)^{-\vartheta-1} s^\nu\d s\\
  &=&\frac{-\vartheta\lambda}{\Gamma(1-\vartheta)\Gamma(\nu+1)} \frac{t^{\nu-\vartheta} \Gamma(-\vartheta)\Gamma(\nu+1)}{\Gamma(1+\nu-\vartheta)}.
 \end{eqnarray*}
 \end{proof}
  \begin{lemma}Consider the Riemann-Liouville fractional integral process \\$\mathcal{N}^{1-\vartheta,\nu}(t),\,\,\vartheta<1,\,0<\nu\leq 1$, we have
 $$ Var[\frac{\d}{\d t}\mathcal{N}^{1-\vartheta,\nu}(t)]=\frac{\lambda t^{\nu-\vartheta}}{\Gamma(1+\nu-\vartheta)}+\frac{\lambda^2}{\Gamma(1+ 2\nu-\vartheta)}\bigg\{2-\frac{\Gamma(1 + 2 \nu)}{\Gamma^2(\nu+1)}\bigg\}t^{2 \nu-\vartheta}.  $$
 \end{lemma}
 \begin{proof}Also from Theorem \ref{theorem:homo-frac}, it follows that
 \begin{eqnarray*}
 Var[\frac{\d}{\d t}\mathcal{N}^{1-\vartheta,\nu}(t)]&=&\frac{-\vartheta}{\Gamma(1-\vartheta)}\int_0^t(t-s)^{-\vartheta-1}Var[N^\nu(s)]\d s\\
 &=&\frac{-\vartheta}{\Gamma(1-\vartheta)}\int_0^t(t-s)^{-\vartheta-1}\\&\times&\bigg\{\frac{2(\lambda s^\nu)^2}{\Gamma(2\nu+1)}-\frac{(\lambda s^\nu)^2}{\Gamma^2(\nu+1)}+\frac{\lambda s^\nu}{\Gamma(\nu+1)}\bigg\}\d s\\
 &=&\frac{\lambda t^{\nu-\vartheta}}{\Gamma(1+\nu-\vartheta)}\\&+&\frac{-\vartheta\lambda^2}{\Gamma(1-\vartheta)}\bigg\{\frac{2}{\Gamma(2\nu+1)}-\frac{1}{\Gamma^2(\nu+1)}\bigg\}\int_0^t(t-s)^{-\vartheta-1} s^{2\nu}\d s\\
  &=&\frac{\lambda t^{\nu-\vartheta}}{\Gamma(1+\nu-\vartheta)}+\frac{-\vartheta\lambda^2}{\Gamma(1-\vartheta)}\bigg\{\frac{2}{\Gamma(2\nu+1)}-\frac{1}{\Gamma^2(\nu+1)}\bigg\}\\&\times&t^{2\nu - \vartheta} \frac{\Gamma(1 + 2 \nu) \Gamma(-\vartheta)}{\Gamma(1+ 2\nu-\vartheta)}
 \end{eqnarray*}
 \end{proof}
 \begin{lemma} For the Weibull's rate function $\Lambda(t)=\big(\frac{t}{b}\big)^a$ and $\vartheta<1$:
 \begin{eqnarray*}
 \E[\frac{\d}{\d t}\mathcal{N}^{1-\vartheta,\nu}_\lambda(t)]&=&\frac{b^{-a\nu}}{\Gamma(\nu+1)}\frac{t^{a\nu-\vartheta}\Gamma(1+a\nu)}{\Gamma(1-\vartheta+a\nu)},\\Var[\frac{\d}{\d t}\mathcal{N}^{1-\vartheta,\nu}_\lambda(t)]&=& \E[\frac{\d}{\d t}\mathcal{N}^{1-\vartheta,\nu}_\lambda(t)]+(2d-q^2)b^{-2a\nu}\frac{t^{2a\nu-\vartheta}\Gamma(1+2a\nu)}{\Gamma(1-\vartheta+2a\nu)},
 \end{eqnarray*} with $q$ and $d$ as given in Theorem \ref{theorem:non-homo-frac}.
 \end{lemma}
 \begin{proof}Now from Theorem \ref{theorem:non-homo-frac}, we have that
 \begin{eqnarray*}
 \E[\frac{\d}{\d t}\mathcal{N}^{1-\vartheta,\nu}_\lambda(t)]&=&\frac{-\vartheta}{\Gamma(1-\vartheta)}\int_0^t(t-s)^{-\vartheta-1}\E[N^\nu_\lambda(s)]\d s \\
 &=&\frac{-\vartheta}{\Gamma(1-\vartheta)}\int_0^t(t-s)^{-\vartheta-1}\frac{\Lambda(s)^\nu}{\Gamma(\nu+1)}\d s\\
 &=&\frac{-\vartheta}{\Gamma(1-\vartheta)\Gamma(\nu+1)}\int_0^t(t-s)^{-\vartheta-1}\bigg(\frac{s}{b}\bigg)^{a\nu}\d s\\
 &=&\frac{-\vartheta b^{-a\nu}}{\Gamma(1-\vartheta)\Gamma(\nu+1)}\int_0^t(t-s)^{-\vartheta-1}s^{a\nu}\d s\\
 &=&\frac{-\vartheta b^{-a\nu}}{\Gamma(1-\vartheta)\Gamma(\nu+1)}\frac{t^{a\nu-\vartheta}\Gamma(-\vartheta)\Gamma(1+a\nu)}{\Gamma(1-\vartheta+a\nu)}
  \end{eqnarray*}
 \end{proof}
 \begin{remark}
 The mean of the non-homogeneous fractional process,\\ $\E[\frac{\d}{\d t}\mathcal{N}^{1-\vartheta,\nu}_\lambda(t)]=\frac{1}{\lambda}\E[\frac{\d}{\d t}\mathcal{N}^{1-\vartheta,\nu}(t)]$ for the Weibull's rate function  for $a=b=1$.
 \end{remark}
\noindent
 For the Gompertz and Makeham's rate functions, we were able to compute the expectations of $\frac{\d}{\d t}\mathcal{N}^{1-\vartheta,\nu}_\lambda(t)$ for $\nu=1$.
 \begin{lemma}Given the Gompertz's rate function $\Lambda(t)=\frac{a}{b}\big(\e^{b t}-1\big)$, we have
 $$\E[\frac{\d}{\d t}\mathcal{N}^{1-\vartheta,1}_\lambda(t)]=a b^{\vartheta-1}\e^{b t}\bigg[1+\frac{\vartheta}{\Gamma(1-\vartheta)}\Gamma(-\vartheta,b t)\bigg]-\frac{a t^{-\vartheta}}{b\Gamma(1-\vartheta)}.$$
 \end{lemma}
 \begin{proof}Following similar steps as above, we obtain
 \begin{eqnarray*}
 \E[\frac{\d}{\d t}\mathcal{N}^{1-\vartheta,1}_\lambda(t)]&=&\frac{-\vartheta a/b}{\Gamma(1-\vartheta)\Gamma(2)}\int_0^t(t-s)^{-\vartheta-1}\big(\e^{b s}-1\big)\d s\\
 &=&\frac{-\vartheta a/b}{\Gamma(1-\vartheta)}\bigg\{\frac{t^{-\vartheta}}{\vartheta}+b^\vartheta\e^{b t}\big[\Gamma(-\vartheta)-\Gamma(-\vartheta,b t)\big]\bigg\}
 \end{eqnarray*}
 \end{proof}
 \begin{lemma}For the Makeham's rate function $\Lambda(t)=\frac{a}{b}\big(\e^{b t}-1+\frac{b\mu}{a}t\big)$, we have
 \begin{eqnarray*}\E[\frac{\d}{\d t}\mathcal{N}^{1-\vartheta,1}_\lambda(t)]&=&\frac{a}{b\vartheta(-1+\vartheta)\Gamma(1-\vartheta)}t^{-\vartheta}\bigg\{-(1-\vartheta)+\frac{b}{a}\mu t \\&+&(1-\vartheta)HypergeometricPFQ\bigg[\{1\},\bigg\{\frac{1}{2}-\frac{\vartheta}{2},1-\frac{\vartheta}{2}\bigg\},\frac{b^2 t^2}{4}\bigg]\\&+&b t HypergeometricPFQ\bigg[\{1\},\bigg\{1-\frac{\vartheta}{2},\frac{3}{2}-\frac{\vartheta}{2}\bigg\},\frac{b^2 t^2}{4}\bigg]\bigg\}.\end{eqnarray*}
 \end{lemma}
 \begin{proof}Now continuing as above, we obtain
 \begin{eqnarray*}
 \E[\frac{\d}{\d t}\mathcal{N}^{1-\vartheta,1}_\lambda(t)]&=&\frac{-\vartheta a/b}{\Gamma(1-\vartheta)\Gamma(2)}\int_0^t(t-s)^{-\vartheta-1}\big(\e^{b s}-1+\frac{b\mu}{a}s\big)\d s\\
 &=&\frac{a/b}{\vartheta(-1+\vartheta)\Gamma(1-\vartheta)}t^{-\vartheta}\bigg\{-1+\vartheta+\frac{b}{a}\mu t \\&-&(-1+\vartheta)HypergeometricPFQ\bigg[\{1\},\bigg\{\frac{1}{2}-\frac{\vartheta}{2},1-\frac{\vartheta}{2}\bigg\},\frac{b^2 t^2}{4}\bigg]\\&+&b t HypergeometricPFQ\bigg[\{1\},\bigg\{1-\frac{\vartheta}{2},\frac{3}{2}-\frac{\vartheta}{2}\bigg\},\frac{b^2 t^2}{4}\bigg]\bigg\}
 \end{eqnarray*}
 \end{proof}
 \begin{remark}For Gomertz and Makeham's rate functions, $$Var[\mathcal{N}^{1-\vartheta,1}_\lambda(t)]=\E[\mathcal{N}^{1-\vartheta,1}_\lambda(t)].$$
 \end{remark}
 \section{Some Auxiliary Results} Here, we will exploit the explicit estimates on the heat kernel for $\alpha$ stable processes. For the condition on the existence and uniqueness result for the stable process, we have:
\begin{theorem}\label{existence-uniqueness-alpha-stable}
Suppose that $\C_{d,\alpha,\beta,\lambda,\nu}<\frac{1}{\lip_\sigma}$ for  positive constant $\lip_\sigma$ together with condition \ref{cond:E-U}, then there exists a random field solution $u$ that is unique up to modification.
\end{theorem}
\noindent
The proof of the above theorem is based on the following Lemma \ref{lemmaA:E-U} and Lemma \ref{lemmaB:E-U}, see Theorem 4.1.1 of \cite{Omaba}.  
 Now let $$\sA u(t,x):=\int_0^t\int_{\R^d} G_{\alpha,\beta}(t-s,x-y)\sigma(u(s,y)D^\vartheta_s N^{\nu}(s)\d y\d s,$$ and
  $$\sA_\lambda u(t,x):=\int_0^t\int_{\R^d}G_{\alpha,\beta}(t-s,x-y)\sigma(u(s,y))D^\vartheta_s N_\lambda^{\nu}(s)\d y\d s,$$ then the following Lemma(s) follow:
\begin{lemma} \label{lemmaA:E-U}
Suppose that u is predictable and $\| u\|_{1,\beta}<\infty$ for all $\beta>0$ and $\sigma(u)$ satisfies condition \ref{cond:E-U}, then $$\|\sA u\|_{1,\beta}\leq {\C}_{d,\alpha,\beta,\lambda,\nu}\lip_\sigma\| u\|_{1,\beta},$$ where ${\C}_{d,\alpha,\beta,\lambda,\nu}:=\frac{2\lambda c_2}{\Gamma(1-\vartheta+\nu)}\frac{d+\alpha}{d+\alpha-1}\frac{\Gamma(\gamma+1)}{\beta^{\gamma+1}}.$
\end{lemma}
\begin{proof}By Lemma \ref{lem:frac}, we have  
\begin{eqnarray*}
\E|{\sA} u(t,x)|&=&\int_0^t\int_{\R^d} G_{\alpha,\beta}(t-s,x-y){\E}|\sigma(u(s,y))|\frac{\lambda s^{\nu-\vartheta}}{\Gamma(1-\vartheta+\nu)}\d y\d s \\
&\leq& \frac{\lambda}{\Gamma(1-\vartheta+\nu)}\\&\times&\int_0^t\int_{\R^d}s^{\nu-\vartheta} G_{\alpha,\beta}(t-s,x-y)\lip_\sigma\E| u(s,y)| \d y\d s.
\end{eqnarray*} Next, Multiply through by $\exp(-\beta t),$ to get
\begin{eqnarray*}
\e^{-\beta t}\E|{\sA} u(t,x)|&\leq&\frac{\lambda\lip_\sigma}{\Gamma(1-\beta+\nu)}\int_0^t\int_{\R^d} s^{\nu-\vartheta}\e^{-\beta(t-s)} G_{\alpha,\beta}(t-s,x-y)\\&\times&\e^{-\beta s}\E| u(s,y)\d y\d s\\
&\leq&\frac{\lambda\lip_\sigma}{\Gamma(1-\vartheta+\nu)}\sup_{s\geq0}\sup_{y\in\R^d}\e^{-\beta s}\E|u(s,y)|\\&\times&\int_0^t\int_{\R^d}s^{\nu-\vartheta}\e^{-\beta(t-s)}G_{\alpha,\beta}(t-s,x-y) \d y\d s.
\end{eqnarray*} 
\noindent
Then we obtain that
\begin{eqnarray*}
\|{\sA}u\|_{1,\beta}&\leq& \frac{\lambda\lip_\sigma}{\Gamma(1-\vartheta+\nu)}\| u\|_{1,\beta}\\&\times&\sup_{t\geq0}\sup_{x\in\R^d} \int_0^t\int_{\R^d}s^{\nu-\vartheta}\e^{-\beta(t-s)}G_{\alpha,\beta}(t-s,x-y) \d y\d s\\
&\leq& \frac{\lambda\lip_\sigma}{\Gamma(1-\vartheta+\nu)}\| u\|_{1,\beta} \int_0^\infty\int_{\R^d}s^{\nu-\vartheta}\e^{-\beta s}G_{\alpha,\beta}(s,y) \d y\d s\\
&\leq&  \frac{\lambda\lip_\sigma}{\Gamma(1-\vartheta+\nu)}\| u\|_{1,\beta}\\&\times& \int_0^\infty\int_{\R^d}s^{\nu-\vartheta}\e^{-\beta s}\bigg\{c_2\bigg(\frac{s^\beta}{|y|^{d+\alpha}}\wedge s^{-\frac{\beta d}{\alpha}}\bigg)\bigg\} \d y\d s.
\end{eqnarray*} The last inequality follows by Lemma \ref{Caputo-fractional}. Let's assume  that $\frac{s^\beta}{| y|^{d+\alpha}}\leq s^{-{\beta d}/\alpha}$ which holds only when $| y|^{\alpha/\beta}\geq s.$ Therefore
\begin{eqnarray*}
\|{\sA}^\alpha u\|_{1,\beta}&\leq&  \frac{\lambda\lip_\sigma c_2}{\Gamma(1-\vartheta+\nu)}\| u\|_{1,\beta}\\&\times&\int_0^\infty  s^{\nu-\vartheta}\e^{-\beta s}\bigg\{s^\beta\int_{| y|\geq s^{\beta/\alpha}}\frac{\d y}{| y|^{d+\alpha}}+s^{-{\beta d}/\alpha}\int_{| y|< s^{\beta/\alpha}}\d y\bigg\}\d s\\
&=&  \frac{\lambda\lip_\sigma c_2}{\Gamma(1-\vartheta+\nu)}\| u\|_{1,\beta}\int_0^\infty  s^{\nu-\vartheta}\e^{-\beta s}\\&\times&\bigg\{s^\beta\bigg(-\int_{-\infty}^{s^{\beta/\alpha}}y^{-(d+\alpha)}\d y
+\int^{\infty}_{s^{\beta/\alpha}}y^{-(d+\alpha)}\d y\bigg)+2s^{{\beta(1-d)}/\alpha}\bigg\}\d s\\
&=& \frac{\lambda\lip_\sigma c_2}{\Gamma(1-\vartheta+\nu)}\| u\|_{1,\beta}\int_0^\infty  s^{\nu-\vartheta}\e^{-\beta s}\\&\times&\bigg\{s^\beta\bigg(-\frac{y^{-(d+\alpha-1)}}{1-d-\alpha}\bigg\arrowvert_{-\infty}^{s^{\beta/\alpha}}+\frac{y^{-(d+\alpha-1)}}{1-d-\alpha}\bigg\arrowvert^{\infty}_{s^{\beta/\alpha}}\bigg)+2s^{{\beta(1-d)}/\alpha}\bigg\}\d s\\
&=&\frac{\lambda\lip_\sigma c_2}{\Gamma(1-\vartheta+\nu)}\| u\|_{1,\beta}\\&\times&\int_0^\infty s^{\nu-\vartheta} \e^{-\beta s}\bigg\{s^\beta\bigg(-\frac{2}{1-d-\alpha}s^{{\beta(1-d-\alpha)}/\alpha}\bigg)+2s^{{\beta(1-d)}/\alpha}\bigg\}\d s\\
&=& \frac{\lambda\lip_\sigma c_2}{\Gamma(1-\vartheta+\nu)}\| u\|_{1,\beta}\\&\times&\int_0^\infty s^{\nu-\vartheta} \e^{-\beta s}\bigg\{\frac{2}{d+\alpha-1}s^{\beta+{\beta(1-d-\alpha)}/\alpha}+2s^{{\beta(1-d)}/\alpha}\bigg\}\d s.
\end{eqnarray*} Thus
\begin{eqnarray*}
\|{\sA} u\|_{1,\beta}&\leq& \frac{2\lambda\lip_\sigma c_2}{\Gamma(1-\vartheta+\nu)}\| u\|_{1,\beta}\frac{d+\alpha}{d+\alpha-1}\int_0^\infty s^{\gamma} \e^{-\beta s}\d s,
\end{eqnarray*} where $\gamma:=\frac{\beta}{\alpha}(1-d)+\nu-\vartheta$. 
Hence
$$\|{\sA} u\|_{1,\beta}\leq \frac{2\lambda\lip_\sigma c_2}{\Gamma(1-\vartheta+\nu)}\| u\|_{1,\beta}\frac{d+\alpha}{d+\alpha-1}\frac{\Gamma(\gamma+1)}{\beta^{\gamma+1}}.$$
\end{proof}
\begin{lemma}\label{lemmaB:E-U}
Suppose  $u$ and $v$ are two predictable random field solutions satisfying $\| u\|_{1,\beta} + \| v\|_{1,\beta}<\infty$ for all $\beta>0$ and $\sigma(u)$ satisfies condition \ref{cond:E-U}, then $$\|{\sA} u-{\sA} v\|_\beta\leq\C_{d,\alpha,\beta,\lambda,\nu}\lip_\sigma\|u-v\|_{1,\beta}.$$
\end{lemma}
\begin{proof}
Similar steps as Lemma \ref{lemmaA:E-U}
\end{proof}
\noindent
We now obtain the following estimates for the Weibull's rate function:
\begin{lemma} \label{lemmaA:E-U-L}
Suppose that u is predictable and $\| u\|_{1,\beta}<\infty$ for all $\beta>0$ and $\sigma(u)$ satisfies assumption \eqref{cond:E-U}, then $$\|\sA_\lambda u\|_{1,\beta}\leq {\C}_{d,\alpha,\beta,\nu}\lip_\sigma\| u\|_{1,\beta},$$ where ${\C}_{d,\alpha,\beta,\nu,a,b}:=\frac{2c_2 b^{-a\nu}}{\Gamma(\nu+1)}\frac{\Gamma(1+a\nu)}{\Gamma(1-\vartheta+a\nu)}\frac{d+\alpha}{d+\alpha-1}\frac{\Gamma(\gamma+1)}{\beta^{\gamma+1}}$ with
$\gamma:=\frac{\beta}{\alpha}(1-d)+a\nu-\vartheta$.
\end{lemma}
\begin{lemma}\label{lemmaB:E-U-L}
Suppose  $u$ and $v$ are two predictable random field solutions satisfying $\| u\|_{1,\beta} + \| v\|_{1,\beta}<\infty$ for all $\beta>0$ and $\sigma(u)$ satisfies condition \ref{cond:E-U}, then $$\|{\sA}_\lambda u-{\sA}_\lambda v\|_\beta\leq\C_{d,\alpha,\beta,\nu,a,b}\lip_\sigma\|u-v\|_{1,\beta}.$$
\end{lemma}
\section{Moment Growths}In this section, we give the proofs of the energy moment growth of our random field solutions. Recall that the mild solution is given by
\begin{eqnarray*}u(x,t)=(\mathcal{G}_t^{\alpha,\beta} u_0)(x)+\sA u(x,t),\end{eqnarray*}
where $$(\mathcal{G}_t^{\alpha,\beta} u_0)(x)=\int_{\R^d}G_{\alpha,\beta}(t,x-y)u(0,y)\d y.$$
We begin with some growth bounds on the semigroup $(\mathcal{G}_t^{\alpha,\beta} u_0)(x)$ and  show that the first term $(\mathcal{G}_t^{\alpha,\beta} u_0)(x)$ of the mild solution grows or decays but only polynomially fast with time. First assume that the initial function $u_0$ is bounded and we have the following:
\begin{lemma}\label{lemma4}There exists some constant $c_0>0$ such that for $\alpha>d,$
\begin{eqnarray*}
|(\mathcal{G}_t^{\alpha,\beta} u_0)(x)|\leq 2c_0\frac{d+\alpha}{d+\alpha-1}t^{\frac{\beta}{\alpha}(1-d)}.
\end{eqnarray*}
\end{lemma}
\begin{proof}Write,
\begin{eqnarray*}
|(\mathcal{G}_t^{\alpha,\beta} u_0)(x)|&=&\bigg\arrowvert\int_{\R^d} G_{\alpha,\beta}(t,x-y)u_0(y)\d y\bigg\arrowvert\\&\leq&\sup_{y\in\R^d}| u_0(y)|\int_{\R^d} G_{\alpha,\beta}(t,x-y)\d y\\
&=& c_0\int_{\R^d} G_{\alpha,\beta}(t,x-y)\d y.
\end{eqnarray*}
Using the estimates on the density of the changed process, for $\alpha>d$:
 \begin{eqnarray*}
|(P_t^\alpha u_0)(x)|\leq c_0\int_{\R^d}\big(t^{-{\beta d}/\alpha}\wedge\frac{t^\beta}{|x-y|^{d+\alpha}}\big)\d y.
\end{eqnarray*}
But
\begin{eqnarray*}
\int_{\R^d}\big(t^{-{\beta d}/\alpha}\wedge\frac{t^\beta}{|x-y|^{d+\alpha}}\big)\d y&=& t^\beta\int_{|x-y|\geq t^{\beta/\alpha}} \frac{\d y}{|x-y|^{d+\alpha}}\\&+&t^{-{\beta d}/\alpha}\int_{|x-y|< t^{\beta/\alpha}}\d y\\
&=&\frac{2}{d+\alpha-1}t^{\beta+\frac{\beta}{\alpha}(1-d-\alpha)}+2t^{\frac{\beta}{\alpha}(1-d)}
\end{eqnarray*}
\end{proof}
\noindent
Next result follows with the assumption that $u_0$ is positive on a set of positive measure.

\begin{proposition}\cite{Foondun}\label{estimatePDEpart}
There exists a $T>0$ and a constant $c_1$ such that for all $t>T$ and all $x\in B(0,\,t^{1/\alpha})$,  $$(\mathcal{G}^{\alpha,\beta}_{t+t_0}u_0)(x)\geq \frac{c_1}{(t+t_0)^{{\beta d}/\alpha}}.$$
\end{proposition}
\subsection{Proofs of main results}
\begin{proof}[ Proof of Theorem \ref{growth-upperbound}]We begin by writing
\begin{eqnarray*}
\E|u(x,t)|&=&|(\mathcal{G}_t^{\alpha,\beta} u_0)(x)|\\&+&\int_0^t\int_{\R^d}G_{\alpha,\beta}(t-s,x-y)\E|\sigma(u(s,y))|\frac{\lambda}{\Gamma(1-\vartheta+\nu)}s^{\nu-\vartheta} \d y \d s\\
&\leq& 2c_0\frac{d+\alpha}{d+\alpha-1}t^{\frac{\beta}{\alpha}(1-d)}+\frac{\lambda}{\Gamma(1-\vartheta+\nu)}\lip_\sigma\\&\times& \int_0^ts^{\nu-\vartheta}\int_{\R^d}G_{\alpha,\beta}(t-s,x-y)\E|u(s,y)| \d y \d s\\
&\leq&c t^{\frac{\beta}{\alpha}(1-d)}+\frac{\lambda\lip_\sigma}{\Gamma(1-\vartheta+\nu)}\\&\times&\int_0^ts^{\nu-\vartheta}\sup_{y\in\R^d}\E|u(s,y)|\int_{\R^d}G_{\alpha,\beta}(t-s,x-y)| \d y\d s
\end{eqnarray*}
Now define $f_{\vartheta,\nu}(t)= t^{\nu-\vartheta}\sup_{x\in\R^d}\E|u(t,x)|$, then 
$$\E|u(x,t)|\leq c t^{\frac{\beta}{\alpha}(1-d)}+\frac{\lambda\lip_\sigma c_2}{\Gamma(1-\vartheta+\nu)}\int_0^t (t-s)^{\frac{\beta}{\alpha}(1-d)}f_{\beta,\nu}(s)\d s.$$
Let $t_0<t<T$ and assume $\frac{\beta}{\alpha}(1-d)>0$. Given that $t-s\leq t<T$, we have
\begin{eqnarray*}
f_{\vartheta,\nu}(t)&\leq& c t^{\nu-\vartheta+\frac{\beta}{\alpha}(1-d)}+\frac{\lambda\lip_\sigma c_2}{\Gamma(1-\vartheta+\nu)} t^{\nu-\vartheta}\int_0^t (t-s)^{\frac{\beta}{\alpha}(1-d)} f_{\vartheta,\nu}(s)\d s\\
&\leq& cT^{\nu-\vartheta+\frac{\beta}{\alpha}(1-d)}+\frac{\lambda\lip_\sigma c_2}{\Gamma(1-\vartheta+\nu)} T^{\nu-\vartheta+\frac{\beta}{\alpha}(1-d)}\int_0^t f_{\vartheta,\nu}(s)\d s.
\end{eqnarray*}Then by Gronwall's inequality, we obtain $$f_{\vartheta,\nu}(t)\leq c_1\exp(c_3 t);\,\,c_1=c T^{\nu-\vartheta+\frac{\beta}{\alpha}(1-d)},\,\,\textrm{and}\,\,c_3=\frac{\lambda\lip_\sigma c_2}{\Gamma(1-\vartheta+\nu)} T^{\nu-\vartheta+\frac{\beta}{\alpha}(1-d)}.$$
\end{proof}
\begin{proof}[ Proof of Theorem \ref{growth-firstmoment}]We begin by taking moment of the solution
\begin{eqnarray*}
\E|u(x,t+t_0)|&=&|(\mathcal{G}^{\alpha,\beta}_{t+t_0}u_0)(x)|+\int_0^{t+t_0}\\&\times&\int_{\R^d}|G_{\alpha,\beta}(t+t_0-s,x-y)|\E|\sigma(u(s,y))|\frac{\lambda}{\Gamma(1-\vartheta+\nu)}s^{\nu-\vartheta} \d y \d s\\
&\geq& c_1(t+t_0)^{-{\beta d}/\alpha}+\frac{\lambda L_\sigma}{\Gamma(1-\vartheta+\nu)}\\&\times& \int_{t_0}^{t+t_0}s^{\nu-\vartheta}\int_{\R^d}||G_{\alpha,\beta}(t+t_0-s,x-y)|\E|u(s,y)| \d y \d s\\
&\geq&  c_1(t+t_0)^{-{\beta d}/\alpha}+\frac{\lambda L_\sigma c_2}{\Gamma(1-\vartheta+\nu)}\\&\times&\int_{t_0}^{t+t_0}s^{\nu-\vartheta}\inf_{y\in B(0,1)}\E|u(s,y)|\int_{B(0,1)}|G_{\alpha,\beta}(t+t_0-s,x-y)| \d s\d y.
\end{eqnarray*}
Make the following change of variable $s-t_0$, then set $v(t,x):=u(t+t_0,x)$ for a fixed $t_0>0$ together with Lemma \ref{estimatePDEpart} to write 
\begin{eqnarray*}
\E|v(x,t)|&\geq& c_1(t+t_0)^{-{\beta d}/\alpha}+\frac{\lambda L_\sigma c_2}{\Gamma(1-\vartheta+\nu)}\\&\times&\int_0^t (s+t_0)^{\nu-\vartheta}\inf_{y\in B(0,1)}\E|v(s,y)|\int_{B(0,1)}(t-s)^{-\beta/\alpha} \d s\d y.
\end{eqnarray*}
Now define $g_{\vartheta,\nu}(t)= (t+t_0)^{\nu-\vartheta}\inf_{x\in B(0,1)}\E|v(t,x)|$ for fixed $t_0>0$, then for $t_0<t<T$,
\begin{eqnarray*}
g_{\vartheta,\nu}(t)&\geq& c_1(t+t_0)^{-\big\{{\beta d}/\alpha+\vartheta-\nu\big\}}\\&+&\frac{\lambda L_\sigma c_2}{\Gamma(1-\vartheta+\nu)} (t+t_0)^{\nu-\vartheta}\int_0^t g_{\vartheta,\nu}(s)\int_{B(0,1)}(t-s)^{-\beta/\alpha} \d s\d y\\
&\geq&  c_1(T+t_0)^{-\big\{{\beta d}/\alpha+\vartheta-\nu\big\}}\\&+&\frac{\lambda L_\sigma c_3}{\Gamma(1-\vartheta+\nu)} (T+t_0)^{\nu-\vartheta} T^{-\beta/\alpha}\int_0^t g_{\vartheta,\nu}(s)\d s
\end{eqnarray*} since $ t_0\leq t\leq T,\,\, 0\leq s<t$ and $t-s\leq t\leq T$. Then we obtain that $g_{\vartheta,\nu}(t)\geq c_4\exp(c_5 t),$ where $c_4=  c_1(T+t_0)^{-\big\{{\beta d}/\alpha+\vartheta-\nu\big\}},$ and $c_5=\frac{\lambda L_\sigma c_3}{\Gamma(1-\vartheta+\nu)} (T+t_0)^{\nu-\vartheta}c_1 T^{-\beta/\alpha}$
\end{proof}
\noindent
The proofs of Theorem \ref{growth-upperbound-non} and Theorem \ref{growth-firstmoment-non} follow from the proofs of the above theorems.
\section{Conclusion}

We observed rather an interesting shift from the usual exponential energy growth bounds for a multiplicative noise perturbation to a class of heat equations. The results showed that the energy growth of the solution is bounded by a product of an algebraic and an exponential functions given by $t^{-(\beta+a\nu)}\exp(c\,t)$, for $c>0$, though the exponential function dominates over the time interval $[t_0,T],\,\,t_0>1$ and $T<\infty$, which causes the solution to behave exponentially. Computational procedure and estimate for the mean and variance for the process for some specific rate functions were given, which can be consequently used for the computation of large variety of physical problems related with non-linear sciences. \\\\
\noindent{\Large\bf Competing Interest}\\\\\
The authors declare that no competing interests exist.
 
\begin{small}
\end{small}
\scriptsize\---------------------------------------------------------------------------------------------------------------------------------------------\\\copyright \it 2017 Omaba; This is an Open Access article distributed under the terms of the Creative Commons Attribution License
(\href{http://creativecommons.org/licenses/by/4.0}{http://creativecommons.org/licenses/by/4.0}),  which permits unrestricted use, distribution, and reproduction in any medium,
provided the original work is properly cited.
\end{document}